\documentclass[]{article}
\usepackage[utf8]{inputenc}
\usepackage{amsmath}
\usepackage{amsfonts}
\usepackage{amssymb}
\usepackage{graphicx}

\usepackage{amsthm}
\theoremstyle{plain}
\usepackage{tikz}
\usetikzlibrary{cd}
\usepackage{mathtools}

\usepackage{hyperref}

\newtheorem{thm}{Theorem}
\newtheorem{definition}[thm]{Definition}
\newtheorem{cor}[thm]{Corollary}
\newtheorem{lem}[thm]{Lemma}
\newtheorem{rem}[thm]{Remark}
\newtheorem{prop}[thm]{Proposition}

\newtheorem{thmintro}{Theorem}
\newtheorem*{conjintro}{Conjecture}

\newcommand{\R}{\mathbb{R}}
\newcommand{\C}{\mathbb{C}}
\newcommand{\Z}{\mathbb{Z}}

\newcommand{\N}{\mathbb{N}}

\newcommand{\cC}{\mathcal{C}}

\newcommand{\del}{\partial}
\newcommand{\delbar}{\bar{\partial}}

\newcommand{\im}{\operatorname{im}}
\newcommand{\pr}{\operatorname{pr}}

\newcommand{\Id}{\operatorname{Id}}
\newcommand{\fg}{\mathfrak{g}}
\newcommand{\fk}{\mathfrak{k}}
\newcommand{\fn}{\mathfrak{n}}
\newcommand{\Lie}{\operatorname{Lie}}

\mathtoolsset{centercolon=true}
\newcommand{\Cdot}{{\raisebox{-0.7ex}[0pt][0pt]{\scalebox{2.0}{$\cdot$}}}}
\mathchardef\mhyphen="2D

\usepackage{authblk}
\title{Fr\"olicher spectral sequence and Hodge structures on the cohomology of complex parallelisable manifolds}
\author{Hisashi Kasuya, Jonas Stelzig}

\begin{document}
\maketitle
\begin{abstract}
\noindent For complex parallelisable manifolds $\Gamma\backslash G$, with $G$ a  solvable or semisimple complex Lie group, the Fr\"olicher spectral sequence degenerates at the second page. In the solvable case, the de-Rham cohomology carries a pure Hodge structure. In contrast, in the semisimple case, purity depends on the lattice, but there is always a direct summand of the de Rham cohomology which does carry a pure Hodge structure and is independent of the lattice.
\end{abstract}

\section{Introduction}
This article is mainly concerned with the cohomology of complex manifolds of the form $X=\Gamma\backslash G$, where $G$ is a connected complex Lie group and $\Gamma\subseteq G$ a cocompact, discrete subgroup. Every such manifold has trivial holomorphic tangent bundle and vice versa Wang \cite{wang_complex_1954}, cf. also \cite[Section 1.12]{winkelmann_complex-analytic_1995}, has shown that every compact complex parallelisable manifold is of this form and we may assume $G$ to be simply connected.
\\

We consider the Fr\"olicher spectral sequence of a compact complex parallelisable manifold $X=\Gamma\backslash G$.
If the Fr\"olicher spectral sequence of a compact complex manifold  degenerates at the first page, then every holomorphic form is closed.
By this fact, we can easily say that the Fr\"olicher spectral sequence of $X$  degenerates at the first page if and only if $G$ is abelian.
Hence, unlike compact K\"ahler manifolds, if $G$ is not abelian     $X$ does not satisfy $\del\delbar$-Lemma as in \cite{DGMS}.\\

If $G$ is nilpotent, the Dolbeault cohomology and de Rham cohomology can be computed from the Lie algebra of $G$ \cite{sakane_compact_1976}. In particular, it is independent of the lattice $\Gamma$. Furthermore, it was recently shown in \cite{popovici_higher-page_2020-1} by Popovici, Ugarte and the second named author that they are page-1-$\del\delbar$-manifolds, meaning that the Fr\"olicher spectral sequence degenerates at the second page and the Hodge filtration induces a pure Hodge structure on the de Rham cohomology. One of the goals of this article is to understand to what extend these results hold for not necessarily nilpotent $G$.\\

Unlike in the nilpotent case, in general the dimensions of Dolbeault and de Rham cohomology can depend on the lattice (see e.g. \cite[Ex. 3.4]{angella_bott-chern_2017} for $G$ solvable and \cite{winkelmann_complex-analytic_1995} or \cite{ghys_deformations_1995} for $G=SL_2(\C)$). 
Nevertheless, it was shown by the first-named author in \cite{kasuya_frolicher_2015} that, for $G$ a solvable complex Lie group, regardless of the choice of $\Gamma$, the Fr\"olicher spectral sequence of degenerates at the second page. Our first result generalizes this by also asserting the existence of a pure Hodge structure on the cohomology of these manifolds.

\begin{thmintro}\label{thm: solv}
	If $G$ is solvable, $X$ is a page-$1$-$\partial\bar\partial$ manifold.
\end{thmintro}

The proof will be given in section \ref{sec: solv}. Similar techniques also yield an analogue of this theorem for certain solvmanifolds of splitting type. (see Thm. \ref{thm: splitting type}).\\

In the semisimple case, the situation is more subtle. Computations by Ghys  \cite{ghys_deformations_1995} and Winkelmann \cite{winkelmann_complex-analytic_1995} show that for $G=\operatorname{SL}_2(\C)$, there exist lattices $\Gamma,\Gamma'\subseteq G$ s.t. $\Gamma\backslash G$ is page-$1$-$\del\delbar$ and $\Gamma'\backslash G$ is not. Building on work of  Akhiezer, we give a conceptual explanation of this phenomenon.

Suppose $G$ is semi-simple.
Take a maximal compact subgroup $K\subset G$.
Denote by $\fg$ and  $\fk$ the Lie algebra of $G$ and $K$ respectively. 
Then, associated with the locally symmetric space $Y=\Gamma\backslash G/K$, we have the canonical injection 
$H^{k}(\fg;\fk,\C)\hookrightarrow H^{k}(\Gamma,\C)$ where $H^{k}(\fg,\fk,\C)$ is the relative Lie algebra cohomology and $H^{k}(\Gamma,\C)$ is the group cohomology (see  \cite[Introduction]{schmid_1981} for instance).
In general, this injection is not an isomorphism.

\begin{thmintro}\label{thm: semsim}
	Let $G$ be semisimple.
	\begin{enumerate}
		\item The Fr\"olicher spectral sequence of $X$ degenerates at the second page.
		\item $X$ is a page-$1$-$\partial\bar\partial$ manifold if and only if $H^{k}(\fg,\fk,\C)\cong  H^{k}(\Gamma,\C)$ for any $k\in \Z$.
	\end{enumerate}
\end{thmintro}

 In view of Theorems \ref{thm: solv} and \ref{thm: semsim}, it appears natural to make the following:
\begin{conjintro}For any compact complex parallelisable manifold $X$, the Fr\"olicher spectral sequence degenerates at the second page.
\end{conjintro}
Apart from the results in this article, indication that this might hold is given by \cite{lescure_annulation_2000}, where it is shown that $d_2:E_2^{0,1}(X)\rightarrow E_2^{2,0}(X)$ always vanishes.

\section{Preliminaries}\label{sec: prelim}
\subsection{Page-$1$-$\partial\bar\partial$-Manifolds}
We briefly recall some of the terminology of \cite{popovici_higher-page_2020-1}.
Let $A=(A^{\Cdot,\Cdot},\del,\delbar)$ be a bounded double complex of complex vector spaces and ${\rm Tot}A^\Cdot:=\bigoplus_{p+q=\Cdot}A^{p,q}$ the associated total complex with differential $d:=\del+\delbar$. Associated with the filtration $F^{r}A=\bigoplus_{p\ge r}A^{p,q}$ we get an induced filtration on the total cohomology $H_{dR}^\Cdot(A):=H^\Cdot ({\rm Tot} A)$, still denoted by $F$, and the spectral sequence 
\[
E_{1}^{p,q}(A)=H_{\delbar}^{p,q}(A)\Longrightarrow (H_{dR}^{p+q}(A),F).
\]
Analogously, for the filtration $\overline{F}^rA=\bigoplus_{q\ge r}A^{p,q}$ we obtain a second spectral sequence converging to $H_{dR}(A)$, which we denote by $\bar{E}_r(A)$. It has the spaces $H_{\del}^{p,q}(A)$ on its first page. The Bott-Chern cohomology is defined as
$$H^{p,q}_{BC}(A):= \frac{\ker\left(\del: A^{p,q} \to A^{p+1,q}\right) \cap \ker\left(\delbar: A^{p,q} \to A^{p,q+1}\right)}{{\rm im}\left(\del\delbar : A^{p-1,q-1} \to A^{p,q}\right)} $$
and the Aeppli cohomology as
$$H^{p,q}_{A}(A):= \frac{{\rm ker}\left(\del\delbar : A^{p-1,q-1} \to A^{p,q}\right))}{{\rm im} \left(\del: A^{p-1,q} \to A^{p,q}\right)+ {\rm im}\left(\delbar: A^{p,q-1} \to A^{p,q}\right)} .$$

All double complexes which we will consider will  have finite dimensional cohomology, i.e. that the spaces $H_{\delbar}^{p,q}(A)$ and $H_{\del}^{p,q}(A)$ are finite dimensional for all $p,q\in\Z$ (this also implies that $E_r^{p,q}(A),\bar E_r^{p,q}(A),H_{BC}^{p,q}(A),H_{A}^{p,q}(A),H_{dR}^k(A)$ are finite dimensional \cite[sect. 2]{stelzig_structure_2018}). Moreover, as the notation already suggests, in most cases we will consider double complexes $A$ that are equipped with a real structure, i.e. an antilinear involution $\overline{(.)}$ s.t. $\overline{ A^{p,q}}=A^{q,p}$ and $\overline{\del\bar{a}}=\delbar a$. In this case the second spectral sequence is determined by the first and can be ignored. We denote by $H^{p,q}(A):=\im (H_{BC}^{p,q}(A)\longrightarrow H_{dR}^{p+q}(A))$. We have equalities\[
H^{p,q}(A)=\{[\alpha]\in H^\Cdot ({\rm Tot} A): \alpha\in A^{p,q}\}=(F^p\cap\bar{F}^{q})H_{dR}^{p+q}(A).
\]
Recall \cite{DGMS} that a double complex $A$ as above is said to satisfy the $\del\delbar$-Lemma (or to have the $\del\delbar$-property) if it satisfies one (hence all) of the following equivalent properties:
\begin{enumerate}
	\item $\ker \del \cap \ker \delbar \cap {\rm im} d={\rm im} \del \delbar $
	\item Both spectral sequences degenerate at page $1$ and for all $k\in\Z$, the filtrations $F$ and $\bar{F}$ on $H^{k}_{dR}(A)$ induce a pure Hodge structure of degree $k$, i.e. $H^k_{dR}(A)=\bigoplus_{p+q=k}H^{p,q}(A)$.
	\item Every class in $H_{\delbar}^{p,q}(A)$ admits a $d$-closed pure-type representative and the map $H_{\delbar}^{p,q}(A)\longrightarrow H^{p,q}(A)$ sending a class via such a representative induces a well-defined isomorphism
\[\bigoplus_{p,q\in\Z} H_{\delbar}^{p,q}(A)\rightarrow H^{\Cdot}_{dR} (A)\] (and analogously for $H_{\del}^{p,q}(A)$).

	\item $A$ is a direct sum of complexes of the following types
	\begin{enumerate}
		\item `squares': complexes with a single pure-bidegree generator $a$, s.t. $\del\delbar a\neq 0$
		\[
		\begin{tikzcd}
		\langle\delbar a\rangle\ar[r]&\langle\del\delbar a\rangle\\
		\langle a\rangle\ar[r]\ar[u]&\langle\del a\rangle\ar[u]
		\end{tikzcd}
		\]
		\item `dots': complexes concentrated in a single bidegree, with all differentials being zero.
		\[\langle a \rangle\]
	\end{enumerate}
\end{enumerate}

The main interest of this paper is following analogue of this property, introduced in \cite{popovici_higher-page_2020-1}.

\begin{definition}
A bounded double complex $A$ is said to have the page-$1$-$\del\delbar$-property if both its spectral sequences degenerate at page $2$ and for all $k\in\Z$, the filtrations $F$ and $\bar{F}$ on $H_{dR}^k(A)$ induce a pure Hodge structure of degree $k$.
\end{definition}
There is an obvious extension of this notion to any page of the Fr\"olicher spectral sequence, including the usual $\del\delbar$-property as the page-$0$-$\del\delbar$-property, but we will not need this here.

Define $h^{p,q}_{\#}(A)=\dim H^{p,q}_{\#}(A)$ and $h^{r}_{\#}(A)=\sum_{r=p+q} h^{p,q}_{\#}(A)$ for each $\#=\del, \delbar, BC, A$.
\begin{prop}[\cite{popovici_higher-page_2020-1}]\label{prop: page-1-ddbar property}
Let $A$ be a bounded  double complex with finite-dimensional cohomology. The following assertions are equivalent:
\begin{enumerate}
	\item  $A$ is page-$1$-$\del\delbar$:
	\item Every class in $E_2^{p,q}(A)$ admits a $d$-closed pure-type representative and the map 
	$ E_{2}^{p,q}(A)\rightarrow H^{p,q}(A)$
	sending a class via such a representative induces a well-defined linear isomorphism
	\[\bigoplus E_{2}^{p,q}(A)\rightarrow H^\Cdot_{dR} (A)\]
	and analogously for $\bar{E}_2$.	
	\item for every $r\in \Z$ the equality $h^{r}_{A}(A)+h^{r}_{BC}(A)=h^{r}_{\delbar}(A)+h^{r}_{\del}(A)$ holds.
	
	\item $A$ is a direct sum of squares, dots and `lines', i.e. complexes generated by a single pure-bidegree generator $a$, s.t. exactly one of $\del a$ or $\delbar a$ is nonzero.
	\[
	\begin{tikzcd}
	\langle \delbar a \rangle\\
	\langle a \rangle \ar[u]
	\end{tikzcd} \qquad	\begin{tikzcd}\langle a\rangle\ar[r]&\langle \del a \rangle	\end{tikzcd}
	\] 
\end{enumerate}
\end{prop}
There is also a characterisation in terms of suitable exactness properties (c.f. \cite{popovici_higher-page_2020-2} for details). 
\begin{rem}\label{rem: page-1-ddbar}
Note that for complexes with real structure property $2$ implies the symmetry $\dim E_2^{p,q}=\dim E_2^{q,p}$. 
\end{rem}

Let $X$ be a complex manifold and  $A_X=(A_X^{\Cdot,\Cdot},\del,\delbar)$ the Dolbeault double complex, i.e. the double complex of $\C$-valued differential forms. 
We denote $E_{r}^{p,q}(X)=E_{r}^{p,q}(A_{X})$ and call it the Fr\"olicher spectral sequence.
In this case the total cohomology is the de Rham cohomology $H^{\Cdot}_{dR}(X,\C)$ of $X$, the $\delbar$-cohomology $H^{p,q}_{\delbar}(A_{X})$ is the Dolbeault cohomology of $X$.
We denote $H^{p,q}_{\#}(X)=H^{p,q}_{\#}(A_{X})$, $h^{p,q}_{\#}(X)=h^{p,q}_{\#}(A_X)$ and $h^{r}_{\#}(X)=h^{r}_{\#}(A_X)$ for each 
 $\#=\del, \delbar, BC, A$.

\begin{definition}\label{def: page-1-ddbar}
Let $X$ be a compact complex manifold. 
$X$ is called a page-$1$-$\del\delbar$-manifold if the Dolbeault double complex
$A_X$ is page-$1$-$\del\delbar$.
\end{definition}

Proposition \ref{prop: page-1-ddbar property} and \cite[Cor. 13]{stelzig_structure_2018} (c.f. also \cite{angella_cohomological_2013}, \cite{angella_bott-chern_2017}, which treat all cases relevant for us) imply the following.

\begin{cor}\label{cor: page-1-ddbar property}
Let $X$ be a compact complex manifold. Given any double complex with $C$ and a map $C\rightarrow A_X$ of double complexes, such that for all $p,q\in\Z$ the induced map $H^{p,q}_{\#}(C)\rightarrow H^{p,q}_{\#}(X)$ is an isomorphism for $\#=\delbar,\del$, then it is also an isomorphism for each $\#=BC, A$. Moreover, if $C$ has the page-$1$-$\del\delbar$-property, then also $X$ is a page-$1$-$\del\delbar$-manifold.
\end{cor}
If $C$ is equipped with a real structure and $\varphi$ respects this map, it is sufficient to consider only the induced map in $H_{\delbar}$.

The following observation gives rise to many examples of page-$1$-$\del\delbar$-complexes.
\begin{lem}[{\rm \cite[Lemma 4.6.]{popovici_higher-page_2020-1}}]\label{lem: tensor product p-1-ddbar}
If $A=(A^\Cdot,d_{1})$, $B=(B^\Cdot,d_{2})$ are (simple) complexes, then the double complex $A\otimes B=(A^\Cdot\otimes B^\Cdot, d_{1}\otimes {\rm Id}_{B},{\rm Id}_{A}\otimes d_{2})$ is page-$1$-$\partial\bar\partial$. The spaces $E_2^{p,q}(A\otimes B)=H^{p,q}(A\otimes B)$ are identified with $H^p(A)\otimes H^q(B)$.
\end{lem}

\subsection{Complex Lie Groups and Lattices}
Let $G$ be a complex Lie group and $\fg$ its Lie algebra. The complex structure of $G$ induces a splitting $\fg_\C=\fg_{1,0}\oplus\fg_{0,1}$ into $i$ and $-i$ Eigenspaces and we denote by $\fg_\C^\vee=\fg^{1,0}\oplus\fg^{0,1}$ the splitting of the dual space. Denote by $\Lambda_\fg:=\Lambda^\Cdot\fg_\C^\vee\subseteq A_G$ the space of left invariant forms on $G$. The restriction of the exterior differential makes it into a sub-double complex of $A_G$. Denote by $\Lambda_\fg^+=(\Lambda^\Cdot \fg^{1,0},\del)$ and $\Lambda_\fg^-=(\Lambda^\Cdot\fg^{0,1},\delbar)$ the simple complexes of left-invariant forms of types $(\Cdot,0)$ and $(0, \Cdot)$. The following observation was made in \cite{popovici_higher-page_2020-1}, (formulated there for $G$ nilpotent):
\begin{lem}\label{lem: li-forms p-1-ddbar}
	The complex $\Lambda_\fg$ is the tensor product of the simple complexes $\Lambda_\fg^+$ and $\Lambda_\fg^-$, hence it is a page-$1$-$\del\delbar$-complex.
\end{lem}

Now assume that there exists a lattice $\Gamma\subseteq G$ (which, for the purpose of this article, will mean that $\Gamma$ is a discrete, cocompact subgroup). We will write $X:=\Gamma\backslash G$ for the compact complex manifold obtained as the quotient of $G$ by the left action of $\Gamma$. Since $\Lambda_\fg$ consists of left-invariant forms, it can also be considered as a sub-double complex of $A_X$. It carries the induced real structure. Hence, by Lemma \ref{lem: li-forms p-1-ddbar} and  Corollary \ref{cor: page-1-ddbar property} we have:
\begin{prop}\label{prop: lie-1deldelbar}
If the inclusion $i:\Lambda_\fg\to  A_X$ induces a cohomology isomorphism $H^{p,q}_{\delbar}(\Lambda_\fg)\cong H^{p,q}_{\delbar}(X)$, then $X$ is  a page-$1$-$\del\delbar$-manifold.
\end{prop}

In \cite{sakane_compact_1976}, Sakane proved that if $G$ is nilpotent, then the inclusion $i:\Lambda_\fg\to   A_X$ induces a cohomology isomorphism $H^{p,q}_{\delbar}(\Lambda_\fg)\cong H^{p,q}_{\delbar}(X)$.
Hence, if $G$ is nilpotent, then  $X$ is a page-$1$-$\del\delbar$-manifold.

If $G$ is not nilpotent  the inclusion  $i:\Lambda_\fg\to  A_X$ does not  induce an isomorphism $H^{p,q}_{\delbar}(\Lambda_\fg)\cong H^{p,q}_{\delbar}(X)$ (e.g. Nakamura manifolds).
Since it is known that Lie-groups admitting a lattice are unimodular \cite[Lem. 6.2]{milnor_curvatures_1976}, by averaging  associated with a bi-invariant Haar measure,  we  obtain a map of double complexes $\mu:  A_{X}^{\ast}\to \Lambda_{\fg}$ such that $\mu\circ i={\rm Id}$ (see the proof of  \cite[Theorem 7]{belgun_2000}). Thus, one obtains a direct sum decomposition $A_X=\Lambda_{\fg}\oplus A_X/\Lambda_{\fg}$ and a corresponding decomposition on all cohomologies considered. In particular, the cohomology groups of $\Lambda_\fg$ inject into those of $X$.
\section{Solvable Case}\label{sec: solv}
We keep the notation of the preceding  section and assume that $G$ is a complex solvable group. 
Let $N$ be the nilradical of $G$.
We can take a connected simply-connected complex nilpotent subgroup $C\subseteq G$  such that $G=C\cdot N$.
Associated with $C$, we define the diagonalizable representation 
\[G=C\cdot N\ni c\cdot n \mapsto ({\rm Ad}_{c})_{s} \in{\rm Aut}(\fg_{1,0})
\] 
 where $({\rm Ad}_{c})_{s}$ is the semi-simple part of the Jordan decomposition of  the adjoint operator $({\rm Ad}_{c}) \in{\rm Aut}(\fg_{1,0})$.
 Denote this representation  by ${\rm Ad}_{s}:G\to {\rm Aut}(\fg_{1,0})$.

We have a basis $ X_{1},\dots,X_{n} $ of $\fg_{1,0}$ such that,
\[ {\rm Ad}_{s}(g) \;=\; {\rm diag} \left(\alpha_{1}(g),\dots,\alpha_{n}(g) \right) \]
for some characters $\alpha_1,\dots,\alpha_n$ of $G$.
Let $B^{\Cdot}_{\Gamma}$ be defined as
\begin{equation*}\label{eq:def-b-holpar}
B^{\Cdot}_{\Gamma} \;:=\; \left\langle \frac{\bar\alpha_{I}}{\alpha_{I} }\, \bar x_{I} \;\middle\vert\; I \subseteq \{1,\ldots,n\} \text{ such that } \left.\left(\frac{\bar\alpha_{I}}{\alpha_{I}}\right)\right\lfloor_{\Gamma}=1 \right\rangle \;, 
\end{equation*}
(where we shorten, {\itshape e.g.} $\alpha_I:=\alpha_{i_1}\cdot\cdots\cdot\alpha_{i_{k}}$ for a multi-index $I=\left(i_1,\dots,i_k\right)$).
Then the inclusion $B^{\Cdot}_{\Gamma}\subset A_{X}^{0,\Cdot}$ induces an isomorphism $H^{\Cdot}(B^{\Cdot}_{\Gamma})\cong H^{0,\Cdot}(X)$ (\cite{kasuya_MRL_2014}).
Consider the following subcomplex of $A_X$:
\begin{equation*}\label{dcplx}
C^{\Cdot_1,\Cdot_2} \;:=\; \Lambda^{\Cdot_1} \fg^{1,0} \otimes_\C B^{\Cdot_2}_{\Gamma} + \overline{ B^{\Cdot_1}_{\Gamma} } \otimes_\C\Lambda^{\Cdot_2} \fg^{0,1} \;.
\end{equation*}
Then $C^{\Cdot_1,\Cdot_2}$ is a double complex with real structure and the inclusion $C^{\Cdot,\Cdot}_{\Gamma}\subset A_{X}^{\Cdot,\Cdot}$ induces an isomorphism
$H^{p,q}_{\delbar}(C^{\Cdot_1,\Cdot_2}_{\Gamma})\cong H^{p,q}_{\delbar}(X)$ (\cite[Section 2.6]{angella_bott-chern_2017}). By Corollary \ref{cor: page-1-ddbar property}), it suffices to show that $C^{\Cdot_1,\Cdot_2}$ is page-$1$-$\del\delbar$ in order to prove Theorem \ref{thm: solv}.

Take the weight decomposition  $\Lambda^{\Cdot} \fg^{1,0}=\bigoplus_{\lambda} V_{\lambda}$ for the representation ${\rm Ad}_{s}:G\to {\rm Aut}(\fg_{1,0})$.
We have $\Lambda^{\Cdot} \fg^{0,1}=\bigoplus_{\lambda} V_{\bar\lambda}$.
(We must distinguish $V_{\bar{1}}\subset \Lambda^{\Cdot} \fg^{0,1}$ from  $ V_{1}\subset \Lambda^{\Cdot} \fg^{1,0}$ for the trivial representation $1$).
Then we can write
\[ \Lambda^{\Cdot_1} \fg^{1,0} \otimes_\C B^{\Cdot_2}_{\Gamma}=\bigoplus_{\lambda}V_{\lambda}\otimes \left(\bigoplus_{\left.\left(\frac{\alpha}{\bar\alpha}\right)\right\lfloor_{\Gamma}=1}\frac{\alpha}{\bar\alpha} V_{\bar\alpha}\right)
\]
and 
\[\overline{ B^{\Cdot_1}_{\Gamma} } \otimes_\C\Lambda^{\Cdot_2} \fg^{0,1}=\bigoplus_{\lambda}\left(\bigoplus_{\left.\left(\frac{\alpha}{\bar\alpha}\right)\right\lfloor_{\Gamma}=1}\frac{\bar\alpha}{\alpha} V_{\alpha}\right)\otimes V_{\bar\lambda}.
\]

We have
\[\Lambda^{\Cdot_1} \fg^{1,0} \otimes_\C B^{\Cdot_2}_{\Gamma}\cap \overline{ B^{\Cdot_1}_{\Gamma} } \otimes_\C\Lambda^{\Cdot_2} \fg^{0,1}=\bigoplus_{\left.\left(\frac{\alpha}{\bar\alpha}\right)\right\lfloor_{\Gamma}=1} \alpha V_{\frac{1}{\alpha}}\otimes \frac{1}{\bar\alpha} V_{\bar\alpha}
\]
and hence 
\[C^{\Cdot_1,\Cdot_2}=\bigoplus_{\left.\left(\frac{\alpha}{\bar\alpha}\right)\right\lfloor_{\Gamma}=1}\left\{ \bigoplus_{\lambda\not=\frac{1}{\alpha}}\alpha V_{\lambda}\otimes \frac{1}{\bar\alpha} V_{\bar\alpha}\oplus \alpha V_{\frac{1}{\alpha}}\otimes \frac{1}{\bar\alpha} V_{\bar\alpha}\oplus \bigoplus_{\lambda\not=\frac{1}{\alpha}} \frac{1}{\alpha}  V_{\alpha}\otimes \bar\alpha V_{\bar\lambda}  \right\}.\]
By ${\rm Ad}_{s}(g) \in {\rm Aut}(\fg_{1,0})$, we have $d=\del : V_{\lambda}\to V_{\lambda}$ and also $d=\delbar : V_{\bar \lambda}\to V_{\bar \lambda}$.
Since we have $\alpha^{-1}d\alpha \in  V_{1}$ for any holomorphic character $\alpha$, we have  $d=\del : \alpha V_{\lambda}\to  \alpha V_{\lambda}$ and  also $d=\delbar : \bar\alpha V_{\bar \lambda}\to  \bar\alpha V_{\bar \lambda}$ for any $\alpha, \lambda$.
Thus, each factor of the direct sum of $C^{\Cdot_1,\Cdot_2}$ is  the tensor product of simple complexes, hence it is a page-$1$-$\del\delbar$-complex by Lemma \ref{lem: tensor product p-1-ddbar}. This implies Theorem \ref{thm: solv}.

\subsection{Variant for Complex Solvmanifolds of Splitting Type}
In this subsection only, $G$ will not be a complex Lie group, but we assume that $G$ is the semi-direct product $\C^{n}\ltimes_{\phi}N$ so that:
\begin{enumerate}
 \item\label{item:ass-1} $N$ is a connected simply-connected $2m$-dimensional nilpotent Lie group endowed with an $N$-bi-invariant complex structure $J_N$; (denote the Lie algebras of $\C^{n}$ and $N$ by $\mathfrak{a}$ and, respectively, $\fn$;)
 \item\label{item:ass-2} for any $t\in \C^{n}$, it holds that $\phi(t)\in {\rm Aut}(N)$ is a holomorphic automorphism of $N$ with respect to $J_N$;
 \item\label{item:ass-3} $\phi$ induces a semi-simple action on $\fn$;
 \item\label{item:ass-4} $G$ has a lattice $\Gamma$; (then $\Gamma$ can be written as $\Gamma = \Gamma_{\C^n} \ltimes_{\phi} \Gamma_{N}$ such that $\Gamma_{\C^n}$ and $\Gamma_{N}$ are  lattices of $\C^{n}$ and, respectively, $N$, and, for any $t\in \Gamma_{\C^n}$, it holds $\phi(t)\left(\Gamma_N\right)\subseteq\Gamma_N$.)

\end{enumerate}
Consider the solvmanifold $X=\Gamma\backslash G$.
This satisfies \cite[Assumption 1.1]{kasuya-mathz-2013} (\cite[Assumption 2.11]{angella_bott-chern_2017}).
\begin{thm}\label{thm: splitting type}
 $X$ is a page-$1$-$\partial\bar\partial$ manifold.
\end{thm}

This result is a generalization of the observation on completely solvable Nakamura manifolds in  \cite[Corollary 3.3]{popovici_higher-page_2020-1}.
\begin{proof}
Consider the standard basis $\left\{ X_{1},\, \dots,\, X_{n} \right\}$ of $\C^n$.
Consider the decomposition $\fn\otimes_\R{\C}=\fn_{1,0}\oplus \fn_{0,1}$ induced by $J_N$.
By the condition {\ref{item:ass-2}}, this decomposition is a direct sum of $\C^{n}$-modules.
By the condition {\ref{item:ass-3}}, we have a basis $\left\{Y_{1},\, \dots,\, Y_{m}\right\}$ of $\fn_{1,0}$ and characters $\alpha_1,\ldots,\alpha_m\in{\rm Hom}(\C^n;\C^*)$ such that the induced action $\phi$ on $\fn_{1,0}$ is represented by
$$ \C^n \ni t \mapsto \phi(t) \;=\; {\rm diag} \left( \alpha_{1}(t),\, \dots,\, \alpha_{m} (t) \right) \in {\rm Aut}(\fn_{1,0}) \;. $$
For any $j\in\{1,\ldots,m\}$, since $Y_{j}$ is an $N$-left-invariant $(1,0)$-vector field on $N$,
the $(1,0)$-vector field $\alpha_{j}Y_{j}$ on $\C^{n}\ltimes _{\phi} N$ is $G$-left-invariant.
Consider the Lie algebra $\fg$ of $G$ and the decomposition $\fg_\C := \fg \otimes_\R \C = \fg_{1,0} \oplus \fg_{0,1}$ induced by $J$.
Hence we have a basis $\left\{ X_{1},\, \dots,\, X_{n},\, \alpha_{1}Y_{1},\, \dots,\, \alpha_{m}Y_{m}\right\}$ of $\fg_{1,0}$, and let $\left\{ x_{1},\, \dots,\, x_{n},\, \alpha^{-1}_{1}y_{1},\, \dots,\, \alpha_{m}^{-1}y_{m} \right\}$ be its dual basis of ${\fg^{1,0}}$.
Then we have 
$$ \Lambda^{p,q}\fg^{\*}_\C \;=\; \begin{array}{c}
\Lambda^{p} \left\langle x_{1},\, \dots,\, x_{n},\, \alpha^{-1}_{1}y_{1},\, \dots,\,\alpha^{-1}_{m}y_{m} \right\rangle \\\otimes \Lambda^{q} \left\langle \bar x_{1},\, \dots,\, \bar x_{n},\, \bar\alpha^{-1}_{1}\bar y_{1},\, \dots,\, \bar\alpha^{-1}_{m}\bar y_{m} \right\rangle\end{array} \;.$$
\medskip
For any $j\in\{1,\ldots,m\}$, there exist unique unitary characters $\beta_{j}\in {\rm Hom}(\C^n;\C^*)$ and $\gamma_{j}\in {\rm Hom}(\C^n;\C^*)$ on $\C^{n}$ such that $\alpha_{j}\beta_{j}^{-1}$ and $\bar\alpha_{j}\gamma^{-1}_{j}$ are holomorphic \cite[Lemma 2.2]{kasuya-mathz-2013}.

Define the differential bi-graded sub-algebra $B^{\Cdot,\Cdot}_{\Gamma}\subseteq A_X^{\Cdot,\Cdot}$ as
\begin{equation*}\label{eq:def-b}
B^{p,q}_{\Gamma} := \C\bigg\langle x_{I} \wedge \left(\alpha^{-1}_{J}\beta_{J}\right)\, y_{J} \wedge \bar x_{K} \wedge \left( \bar\alpha^{-1}_{L}\gamma_{L} \right) \, \bar y_{L}\bigg\rangle_{
\left\{
\substack{|I| + |J|=p \text{ and } |K| + |L| = q\\
 \text{ such that } \left( \beta_{J}\gamma_{L} \right)\lfloor_{\Gamma} = 1}
 \right\}}
 \;.
\end{equation*}
Define 
$$C_{\Gamma}^{\Cdot, \Cdot}:= B^{\Cdot,\Cdot}_{\Gamma}+\overline{B^{\Cdot,\Cdot}_{\Gamma}}.
$$
Then $C_{\Gamma}^{\Cdot, \Cdot}$ is a double complex with real structure and the inclusion $C^{\Cdot,\Cdot}_{\Gamma}\subset A_{X}^{\Cdot,\Cdot}$ induces an isomorphism
$H^{p,q}_{\delbar}(C^{\Cdot,\Cdot}_{\Gamma})\cong H^{p,q}_{\delbar}(X)$ and $H^{\Cdot,\Cdot}_{BC}(C^{\Cdot,\Cdot}_{\Gamma})\cong H^{\Cdot,\Cdot}_{BC}(X)$ (\cite[Section 2.5]{angella_bott-chern_2017}). Again, it suffices to show that $C_\Gamma^{\Cdot,\Cdot}$ is page-$1$-$\del\delbar$.

Take the weight decomposition  $\Lambda^{\Cdot} \fn^{1,0}=\bigoplus_{\alpha} V^{1,0}_{\alpha}$ and  $\Lambda^{\Cdot} \fn^{0,1}=\bigoplus_{\bar\alpha} V^{0,1}_{\bar\alpha}$  for the  $\C^{n}$-action.
Then we have 
\[ B^{\Cdot,\Cdot}_{\Gamma}=\bigoplus_{\left( \beta\gamma \right)\lfloor_{\Gamma} = 1} \Lambda^{\Cdot,\Cdot}\C^{n}\otimes (\alpha\beta^{-1} V^{1,0}_{\alpha})\otimes (\bar\lambda\gamma^{-1} V^{0,1}_{\bar\lambda})
\]
where $\beta\in {\rm Hom}(\C^n;\C^*)$ and $\gamma \in {\rm Hom}(\C^n;\C^*)$ are unitary characters on $\C^{n}$ such that $\alpha\beta^{-1}$ and $\bar\lambda\gamma^{-1}$ are holomorphic.
We have
\begin{equation*}
\begin{split}
C_{\Gamma}^{\Cdot, \Cdot}=\bigoplus_{\left( \beta\gamma \right)\lfloor_{\Gamma} = 1, \alpha\bar\lambda\not=\beta\gamma} \Lambda\C^{n}\otimes (\alpha\beta^{-1}\bar\lambda\gamma^{-1} V^{1,0}_{\alpha})\otimes  V^{0,1}_{\bar\lambda}\\
\oplus \bigoplus_{\left( \beta\gamma \right)\lfloor_{\Gamma} = 1, \alpha\bar\lambda\not=\beta\gamma} \Lambda\C^{n}\otimes  V^{1,0}_{\lambda}\otimes \left(\overline{\alpha\beta^{-1}\bar\lambda\gamma^{-1}}  V^{0,1}_{\bar\alpha}\right) \\
\oplus  \bigoplus_{\left( \alpha\bar\lambda \right)\lfloor_{\Gamma} = 1
} \Lambda \C^{n}\otimes  V^{1,0}_{\alpha}\otimes V^{0,1}_{\bar\lambda}.
\end{split}
\end{equation*}
Since $J_N$ is bi-invariant, we have $d=\del : V^{1,0}_{\alpha}\to V^{1,0}_{\alpha}$ and $d=\delbar : V^{0,1}_{\bar\alpha}\to V^{0,1}_{\bar\alpha}$.
Since $\alpha\beta^{-1}\bar\lambda\gamma^{-1} $ is holomorphic, we have \[d=\del :\alpha\beta^{-1}\bar\lambda\gamma^{-1}  V^{1,0}_{\alpha}\to \Lambda^{1,0}\C^{n}\otimes (\alpha\beta^{-1}\bar\lambda\gamma^{-1} V^{1,0}_{\alpha})\] and \[d=\delbar :\overline{\alpha\beta^{-1}\bar\lambda\gamma^{-1}}  V^{0,1}_{\bar\alpha} \to  \Lambda^{0,1}\C^{n}\otimes  (\overline{\alpha\beta^{-1}\bar\lambda\gamma^{-1}}  V^{0,1}_{\bar\alpha}).\]
Thus each  each factor of the direct sum of $C^{\Cdot,\Cdot}$ is  the tensor product of simple complexes, hence it is a page-$1$-$\del\delbar$-complex.\end{proof}

\section{Semisimple Case}\label{sec: semsim}
In this section, we keep the notation from section \ref{sec: prelim} and assume $G$ to be a complex semisimple Lie group, and $K\subseteq G$ a maximal compact subgroup. Denote by  $\fg=\Lie(G),\fk=\Lie(K)$ the respective Lie algebras. Note that we have $\fg=\fk\oplus J\fk$.\\

As before, we denote by $X:=\Gamma \backslash G$ the quotient by the left action of $\Gamma$. Later, we will also consider $Y:=G/K$ the quotient of the right action by $K$ and $Z:=\Gamma\backslash G/K$ the quotient by both. A first observation is:

\begin{lem}\label{lem: maximal compact subgroup cohomology}
	There is a canonical isomorphism $H(\fg_{1,0})=H_{dR}(K,\C)$.
\end{lem}

\begin{proof}
	This is a consequence of the identification of complex Lie algebras $\fg_{1,0}\cong \fg = \fk\oplus J\fk=\fk_\C$ (whence the identification $\Lambda^\Cdot \fg^{1,0}\cong \Lambda^\Cdot\fk_\C^\vee$) and the fact that the de Rham cohomology of compact Lie groups can be computed from left-invariant forms.
\end{proof}

\begin{thm}\label{thm: deg FSS ss-case}
	The Fr\"olicher spectral sequence of $X$ degenerates at page $2$, with $E_{2}^{p,q}=H^p_{dR}(K,\C)\otimes H^q(\Gamma,\C)$.
\end{thm}

We will give two proofs of this theorem, each one highlighting somewhat different aspects. Both will use the following computation of Dolbeault cohomology by Akhiezer.
Note that the left-multiplication of $G$ on $X$ induces a holomorphic representation of $G$ on the Dolbeault cohomology $H_{\delbar}^{p,q}(X)$.

\begin{thm}[Akhiezer \cite{akhiezer_group_1997}]\label{thm: Akhiezer's description of Dolbeault cohomology}
There is a $G$-module isomorphism
	\[H_{\delbar}^{p,q}(X)=\Lambda^p \fg^{1,0}\otimes H^q(\Gamma,\C)\]
where $\Lambda^p \fg^{1,0}$ is a $G$-module induced by the adjoint representation and $H^q(\Gamma,\C)$ is a trivial $G$-module.
\end{thm}

\begin{proof}[Proof of Theorem \ref{thm: deg FSS ss-case} via 4 spectral sequences]\mbox{}

{\em Claim 1.  The Hochschild-Serre spectral sequence for  $\fk\subset \fg$ degenerates  at the second page.} 
We notice that $H^{\ast}( \fg_\C)=H^{\ast}(\fg_{1,0})\otimes H^{\ast}(\fg_{0,1})\cong H^{\ast}(\fk_{\C})\otimes \overline{H^{\ast}(\fk_{\C})}$ by Lemma \ref{lem: maximal compact subgroup cohomology}.
Let $E^{p,q}_{r (HS)}$ be the Hochschild-Serre spectral sequence for  $\fk\subset \fg$.
We have  $E^{p,q}_{2 (HS)}=H^{p}(\fg; \fk, \C)\otimes H^{q}( \fk_{\C})=\left(\Lambda^{p}(J\fk)_{\C}^{\ast}\right)^{\fk}\otimes H^{q}( \fk_{\C})$ by the decomposition $\fg = \fk+J\fk$ with $[J\fk, J\fk]\subset \fk$ and $[\fk, J\fk]\subset J\fk$.
Since $J$ gives a $\fk$-module isomorphism $\fk\cong J\fk$, we have 
$H^{\ast}(\fk_{\C})\cong \left(\Lambda^{p}(\fk)^{\ast}_{\C}\right)^{\fk} \cong \left(\Lambda^{p}(J\fk)_{\C}^{\ast}\right)^{\fk}$.
Thus the claim follows.

{\em Claim 2.  The Serre spectral sequence for the fiber bundle    $X\to Z$ degenerates at the second page.} 	Let $E^{p,q}_{r (S)}$ be the Serre spectral sequence for the fiber bundle    $X\to Z$.
Since  $X\to Z$ is a principal $K$-bundle, we have $E^{p,q}_{2 (S)}=H^{p}_{dR}(Z)\otimes H^{q}_{dR}(K)$.
By the multiplicative structure on the spectral sequence, we have $d_{2}^{p,q}={\rm Id}_{H^{p}_{dR}(Z)}\otimes d_{2}^{0,q}$.
We notice that the inclusion $i:\Lambda \fg^{\ast}\subset A_{X}^{\ast}$ induces  a morphism $i^{p,q}_{r}:E^{p,q}_{r(HS)}\to E^{p,q}_{r(S)}$.
By the averaging, we have a map $\mu:  A_{X}^{\ast}\to \Lambda \fg^{\ast}$ such that $\mu\circ i={\rm Id}$.
This implies that $i^{p,q}_{r}:E^{p,q}_{r(HS)}\to E^{p,q}_{r(S)}$ is injective.
By $E^{0,q}_{r(HS)}\cong H^{q}(\fk)\cong H^{q}_{dR}(K)$, we can say that $i^{0,q}_{2}:E^{0,q}_{r(HS)}\to E^{0,q}_{r(S)}$ is an isomorphism.
By the first claim, we have $d_{2}^{0,q}=0$ on $E^{0,q}_{2(HS)} $ and hence $d_{2}^{0,q}=0$ on $E^{0,q}_{2(S)} $.
This implies $d_{2}^{p,q}=0$ on $E^{p,q}_{2(S)} $.
By the same argument, we can say $d_{r}^{p,q}=0$ on $E^{p,q}_{r(S)} $ for any $r\ge 2$.

{\em Claim 3. The Fr\"olicher spectral sequence of $X$ degenerates at the second page}.
Let $E^{p,q}_{r (FS)}$ be the Serre spectral sequence for the fiber bundle    $X\to Z$.
By Akhiezer's theorem, we have $E^{p,q}_{1 (FS)}=H^{p,q}_{\delbar}(X)=\Lambda^p \fg^{1,0}\otimes H^q(\Gamma,\C)$.
The differential $d_{1}^{p,q}$ on  $E^{p,q}_{1 (FS)}= H^{p,q}_{\delbar}(X)$ is induced by the differential $\del$ 
and we can easily check that it is determined by the $\fg_{1,0}$-module structure of $H^{0,q}_{\delbar}(X)=H^q(\Gamma,\C)$.
Since $H^q(\Gamma,\C)$ is a trivial $G$-module,  $\fg_{1,0}$ acts trivially on $H^{0,q}(X)$ and 
we can say that the differential $\del$ induces a trivial differential on $H^q(\Gamma,\C)$.
Thus we have $E^{p,q}_{2 (FS)}\cong H^p(\fg_{1,0})\otimes H^q(\Gamma,\C)$.
By the above argument, we have $E^{p,q}_{2 (S)}\otimes \C=H^{p}(Z,\C)\otimes H^{q}(K,\C)\cong  H^p(\Gamma,\C)\otimes H^{q}(\fk_{\C})\cong H^p(\Gamma,\C)\otimes H^{q}( \fg_{1,0})$.
By the $E_{2}$-degeneration of $E^{p,q}_{r(S)} $, for any $r\in \Z$
we have $\sum_{p+q=r}\dim E^{p,q}_{2 (FS)}=\sum_{p+q=r}\dim E^{p,q}_{2 (S)}=\dim H^{r}_{dR}(X)$ and hence the claim follows.
\end{proof}

\begin{proof}[Proof of Theorem \ref{thm: deg FSS ss-case} via lattice-cohomology subcomplex of $A_X$]
	The lattice cohomology $H(\Gamma,\C)$ can be computed via the complex $A_Y^\Gamma$ of $\Gamma$-invariant forms on $Y$ \cite[Ch. VII]{borel_continuous_2000}. By pulling forms back to $G$ then pushing forward to $X$, we obtain a map of (simple) complexes
	\[
	\sigma: A_Y^\Gamma\longrightarrow A_X
	\]
	We will also consider the projection map 
	\[
	\pr: A_X\rightarrow (A^{0,\Cdot}_X,\delbar)
	\]
	which is a map of complexes and induces the `edge maps' 
	\[
	H_{dR}^q(X)\rightarrow H_{\delbar}^q(X)
	\]
	{\em Claim: The composition $\pr\circ\sigma$ induces isomorphisms $H^q(\Gamma,\C)\cong E_1^{0,q}=H^q(\Gamma,\C)$.}\\
	
	Admitting the claim, the edge maps are surjective, hence there can be no differentials starting at $E_r^{0,q}$, for any $q$ and $r\geq 1$. Thus, by the Leibniz-rule, any possible nonzero differential has to live on $E_2$ and be of the form $\Id\otimes d_{\fg^{1,0}}$.\\
	
	In order to prove the claim, we identify (c.f. \cite[Ch. VII]{borel_continuous_2000}):
	\begin{enumerate}
		\item $A_Y^\Gamma$ with the $K$-invariants of $A:=\Lambda(\fg/\fk)^\vee_\C\otimes \cC^\infty (X,\C)$, where the action of $K$ is induced by right multiplication on $X$ on the right factor and the adjoint action on the left factor,
		\item $A_X$ with $\Lambda_\fg\otimes \cC^\infty(X,\C)$,
		\item $A^{0,\Cdot}_X$ with $\Lambda^-_\fg\otimes \cC^\infty(X,\C)$.
	\end{enumerate}

Under these identifications, consider the map $\tilde{\sigma}: A\rightarrow A_X$ induced by the inclusion $(\fg/\fk)_\C^\vee\rightarrow \fg^\vee_\C$. Since the composition $\fg_{0,1}\rightarrow \fg_\C\rightarrow (\fg/\fk)_\C$ is an isomorphism, so is $\pr\circ\tilde\sigma$. On the other hand, $\pr\circ\tilde\sigma$ is an equivariant map, therefore it identifies the $K$-invariants on the left (i.e. $A_Y^\Gamma$) with the $K$-invariants on the right. But by averaging over $K$, one sees that the complex of $K$-invariants is a direct summand in $A^{0,\Cdot}$. Therefore, the edge map $H^q(\Gamma,\C)\rightarrow E_1^{0,q}$ induced by $\pr\circ\sigma$ has to be injective. Since by Akhiezer's result both source and target have the same dimension, the claim follows.
\end{proof}

We consider  the canonical injection 
$H^{k}(\fg;\fk,\C)\hookrightarrow H^{k}(\Gamma,\C)$  associated with the locally symmetric space $Z:=\Gamma\backslash G/K$. In view of Theorem \ref{thm: Akhiezer's description of Dolbeault cohomology} it can be identified with the map induced by the inclusion of left-invariant forms $H^{0,k}_{\delbar}(\Lambda_\fg)\hookrightarrow H_{\delbar}^{0,k}(X)$.

\begin{cor}
	$X$ is a page-$1$-$\partial\bar\partial$ manifold if and only if $H^{k}(\fg,\fk,\C)\cong  H^{k}(\Gamma,\C)$ for any $k\in \Z$. 
\end{cor}
\begin{proof}
Suppose $X$ is a page-$1$-$\partial\bar\partial$ manifold.
By Remark \ref{rem: page-1-ddbar}, we have $\dim E_{2}^{p,q}=\dim E_{2}^{q,p}$.
By Theorem \ref{thm: deg FSS ss-case}, this implies $\dim H^{k}_{dR}(K)=\dim H^{k}(\Gamma)$ for any $k\in \Z$.
As we saw in the first proof of Theorem \ref{thm: deg FSS ss-case}, we have $H^{k}(\fg;\fk)\cong H^{k}_{dR}(K)$.
Thus we have $\dim H^{k}(\fg;\fk,\C)=\dim H^{k}(\Gamma,\C)$ and so the injection 
$H^{k}(\fg;\fk,\C)\hookrightarrow H^{k}(\Gamma,\C)$ is an isomorphism.

We assume $H^{k}(\fg,\fk,\C)\cong  H^{k}(\Gamma,\C)$ for any $k\in \Z$.
Then by Theorem \ref{thm: Akhiezer's description of Dolbeault cohomology}
 and Lemma \ref{lem: maximal compact subgroup cohomology},
 we have
 \[\dim H_{\delbar}^{p,q}(X)=\dim \Lambda^p \fg^{1,0}\otimes H^q(\fg_{1,0})=\dim H_{\delbar}^{p,q}(\Lambda_\fg).
 \]
Hence the injection  $H_{\delbar}^{p,q}(\Lambda_\fg)\hookrightarrow H_{\delbar}^{p,q}(X)$ is an isomorphism and so $X$ is a page-$1$-$\partial\bar\partial$ by Proposition \ref{prop: lie-1deldelbar}.\end{proof}
\begin{rem}
For $G=\operatorname{SL_2(\C)}$ with $K=\operatorname{SU(2)}$, $b_{1}(K)=b_{2}(K)=0$ but for any $r\in\N$, there exist lattices $\Gamma\subseteq G$ such that $b_1(\Gamma)=b_2(\Gamma)=r$ (see \cite[§B]{winkelmann_complex-analytic_1995}, \cite[6.2.]{ghys_deformations_1995}).  Note that by a theorem of Raghunathan \cite{raghunathan_vanishing_1966}, $b_1(\Gamma)=0$ as soon as $G$ has no $\operatorname{SL}_2(\C)$-factor.
Moreover, for $k<{\rm rk}_{\R}\fg$, the injection $H^{k}(\fg;\fk,\C)\hookrightarrow H^{k}(\Gamma,\C)$ is an isomorphism(see \cite[VII, Corollary 4.4]{borel_continuous_2000}).
\end{rem}


\bibliographystyle{acm}

\end{document}